\documentclass[12 pt]{article}
\usepackage{amsmath,amsfonts,amsthm,amssymb,array,mathrsfs,latexsym,paralist,  color}
\usepackage[all]{xy}
\usepackage{tikz-cd, soul, pgfplots}
\usepackage{lipsum}

\usepackage{dsfont}

\usetikzlibrary{arrows,automata}

\tikzstyle{V}=[fill=black,circle,scale=0.2, outer sep = 4pt]

\newtheorem{thm}{Theorem}[section]
\newtheorem{prop}[thm]{Proposition}
\newtheorem{proposition}[thm]{Proposition}

\newtheorem{lemma}[thm]{Lemma}
\newtheorem{thmx}{Theorem}

\theoremstyle{remark}
\newtheorem{rmk}[thm]{Remark}

\theoremstyle{definition}
\newtheorem{defn}[thm]{Definition}

\numberwithin{equation}{section}

\newcommand\numberthis{\addtocounter{equation}{1}\tag{\theequation}}

\DeclareMathOperator{\Aut}{Aut}

\DeclareMathOperator{\Ind}{Ind}
\DeclareMathOperator{\Res}{Res}

\DeclareMathOperator{\gMod}{\mathbf{G-Mod}}

\DeclareMathOperator{\KK}{KK}
\DeclareMathOperator{\KO}{KO}
\DeclareMathOperator{\KU}{KU}
\DeclareMathOperator{\KT}{KT}
\DeclareMathOperator{\KX}{KX}
\DeclareMathOperator{\KY}{KY}
\DeclareMathOperator{\Kk}{K}
\DeclareMathOperator{\Hh}{H}
\DeclareMathOperator{\ABC}{ABC}

\newcommand{\bi}{\begin{itemize}}
\newcommand{\ei}{\end{itemize}}
\newcommand{\be}{\begin{enumerate}}
\newcommand{\ee}{\end{enumerate}}

\newcommand{\C}{\mathbb{C}}

\newcommand{\K}{\mathcal{K}}

\newcommand{\R}{\mathbb{R}}

\newcommand{\Z}{\mathbb{Z}}

\newcommand{\id}{\operatorname{id}}

\def\sp{^}

\def\pr{^{\scriptscriptstyle \R}}
\def\pc{^{\scriptscriptstyle \C}}
\def\po{^{\scriptscriptstyle \mathrm{O}}}
\def\pu{^{\scriptscriptstyle \mathrm{U}}}
\def\pt{^{\scriptscriptstyle \mathrm{T}}}
\def\px{^{\scriptscriptstyle \mathrm{X}}}
\def\py{^{\scriptscriptstyle \mathrm{Y}}}
\def\sr{_{\scriptscriptstyle \R}}
\def\sc{_{\scriptscriptstyle \C}}
\def\crr{^{\scriptscriptstyle {\mathrm{CR}}}}
\def\crt{^{\scriptscriptstyle {\mathrm{CRT}}}}

\def\CRT{\mathrm{CRT}}

\newcommand{\cs}{\mathrm{C}^\ast}

\newcommand{\Input}{\mathop{\underline{\quad}}}
\newcommand{\suchthat}{\;\ifnum\currentgrouptype=16 \middle\fi|\;}

\begin{document}

\title{Functoriality of real crossed product K-theory {spectral sequences} with respect to group homomorphisms}
\author{Jeffrey L. Boersema, Sarah L. Browne, Elizabeth Gillaspy,\\ and Alistair Miller}

\maketitle
\begin{abstract}
    Spectral sequences are a key tool for computing the $\Kk$-theory of a crossed product $\cs$-algebra. 
    However, the impact of a group homomorphism $\Omega\colon G \to H$ on such a spectral sequence was unknown until quite recently, even when $G = \mathbb Z^\ell$, $H = \mathbb Z^{k}.$ Recent work [Mil25] of the fourth-named author in the complex case establishes that ABC spectral sequences are functorial with respect to group homomorphisms. In this paper, we obtain the analogous result for real $\Kk$-theory and for united $\Kk$-theory. Specifically,  we first show that the ABC spectral sequence approximates $\KO_*(G \ltimes_r A)$ with the group homology $\Hh_p(G;\KO_q(A))$ when $G$ is a torsion-free discrete group satisfying the Baum--Connes conjecture with coefficients in $A$. Then, for a homomorphism $\Omega \colon G \to H$ of such groups with amenable kernel, and a real $H$-$\cs$-algebra $A$,
   we show  moreover that the map in K-theory induced by the $*$-homomorphism $G \ltimes_r A \to H \ltimes_r A$ is approximated by the natural map in group homology.
\end{abstract}

\section{Introduction}
$\Kk$-theory is a key $\cs$-algebraic invariant, yet its computation is nontrivial, in both the real and complex settings.  
In this paper, we are concerned with the map $\KO_*(G \ltimes_r A) \to \KO_*(H \ltimes_r A)$ on the $\Kk$-theory of {real} crossed product $\cs$-algebras which is induced by a homomorphism $\Omega \colon G \to H$ of discrete groups (with amenable kernel).  

Even in the most straightforward case when $G = \Z^\ell
$, our best understanding of the $\Kk$-theory of the (real or complex) crossed-product $\cs$-algebra $G \ltimes A$ passes through a spectral sequence. (See Appendix \ref{sec:appendix} for a discussion of this.) Until quite recently, the question of the functoriality of the spectral sequence with respect to group homomorphisms had not been addressed in the literature. This defect was recently remedied (in a very general setting) by the fourth-named author \cite[Theorems F and 6.1]{miller-hlogy}. 
Roughly speaking, \cite[Theorem 6.1]{miller-hlogy} tells us that if one has the input data needed to construct two {ABC} spectral sequences, and a map (the ABC morphism) between these data, then the ABC morphism yields a morphism of the spectral sequences and a compatible morphism of the target groups.

Although \cite[Theorem 6.1]{miller-hlogy} holds in great generality, its applicability in concrete situations depends on a thorough understanding of {what} the ABC morphism {does} on the $E^2$ pages.  For many situations involving complex $\cs$-algebras and group homomorphisms, \cite{miller-hlogy} provides these details. 
In particular, using \cite[Example 3.8]{miller-categories}, from a homomorphism $\Omega\colon G \to H$ of torsion-free countable discrete groups satisfying the Baum--Connes conjecture with coefficients in a complex $H$-$\cs$-algebra $A$, we obtain a morphism between the spectral sequences converging to $\Kk_*(G \ltimes_r A)$ and $\Kk_*(H \ltimes_r A)$. Indeed, by \cite[Example 5.11]{miller-categories}, when $\ker \Omega$ is amenable this morphism converges to the map in K-theory induced by the $*$-homomorphism $G \ltimes_r A \to H \ltimes_r A$. Moreover, by \cite[Remark 4.6 and Theorem 5.14]{miller-hlogy}, the  $E^2$ pages of these spectral sequences are given by group homology with coefficients in $\Kk_*(A)$, and the morphism on the $E^2$ page is precisely the map $\Omega_*\colon \Hh_p(G; \Kk_q(A)) \to \Hh_p(H; \Kk_q(A))$ induced by $\Omega$ on group homology.

Our main goal in this paper is to establish the analogous result for crossed products of real $\cs$-algebras.  This main result is phrased in terms of united $\Kk$-theory as in \cite{boersema2002, boersema2004, brs}, of which we provide a quick review in Section~\ref{subsection-unitedK}. 
 The united $\Kk$-theory $\Kk\crt(A)$ of a real $\cs$-algebra $A$ is a system of abelian groups and natural transformations among them, known as a $\CRT$-module, incorporating real, complex, and self-conjugate $\Kk$-theory.
In our setting, 
we approximate $\Kk\crt(G \ltimes_r A)$ by
a system of three spectral sequences and natural transformations between them, in which the differential maps commute with the natural transformations.
We call such a system a $\CRT$-spectral sequence; see Section \ref{subsection-unitedK} for details.

 In our main result, Theorem~\ref{intro main theorem}, we articulate the naturality of 
 this $\CRT$-spectral sequence with respect to group homomorphisms.

\begin{thmx}[See Theorem \ref{CRT-sp.seq.summary}]\label{intro main theorem}
Let $H$ be a discrete, countable, torsion-free group which satisfies the Baum--Connes conjecture with coefficients in a separable real $H$-$\cs$-algebra $A$. Then there is a $\CRT$-spectral sequence $E^r_{p,q} = (E^r_{p,q})\crt(H,A)$ 
which converges to $\Kk_*\crt(H \ltimes_r A)$. Moreover, this is functorial; given a homomorphism $\Omega \colon G \to H$ of discrete, countable, torsion-free groups, whose kernel is amenable,
there is a morphism of $\CRT$-spectral sequences
\[ (E^r_{p,q})\crt(G,A) \to (E^r_{p,q})\crt(H,A) \]
which is given by 
\[\Hh_p(\Omega,\Kk_q\crt(A)) \colon \Hh_p(G;\Kk_q\crt(A)) \to \Hh_p(H;\Kk_q\crt(A)) \]
on the $E^2$ page and converges to \[ \Kk_q\crt(\Omega \ltimes_r \id_A) \colon \Kk_q\crt(G \ltimes_r A) \to \Kk_q\crt(H \ltimes_r A)\] on the $E^\infty$ page.
\end{thmx}

These hypotheses on the groups $G, H$ are needed for a number of reasons. Countability of the groups, like separability of the $\cs$-algebras, is needed to access the technical power of Kasparov's $\KK$-theory. 
Working with torsion-free discrete groups guarantees that the $E^2$ page of the spectral sequence is given by the homology groups $\Hh_p(G;\Kk_q\crt(A))$ and $\Hh_p(H;\Kk_q\crt(A))$ 
(see Remark \ref{rmk:torsion-free} below).   Assuming that $G, H$ satisfy the Baum--Connes conjecture with coefficients guarantees that the  ABC spectral sequences converge to the operator K-theory groups $\Kk_q(G \ltimes_r A)$ and $\Kk_q(H \ltimes_r A)$ (see Remark \ref{rmk:full-vs-reduced}). Finally, as we show in Proposition \ref{prop:amenable-kernel}, the hypothesis that $\ker \Omega$ is amenable guarantees the existence of the $*$-homomorphism $\Omega \ltimes_r \id_A$,  and hence gives us an identifiable expression for the impact of $\Omega$ on the $E^\infty$ pages.

Our interest in 
the functoriality of real $\cs$-algebraic $\Kk$-theory with respect to group homomorphisms  
arose from a desire to relate the $\Kk$-theory of a (real) higher-rank graph $\cs$-algebra $C^*\sr(\Lambda)$ with the $\Kk$-theory of the underlying directed graphs.  (If a higher-rank graph $\Lambda$ has rank $k$, we say $\Lambda$ is a $k$-graph; a directed graph is then a 1-graph.) For a $k$-graph $\Lambda$,  \cite[Theorem 3.1]{boersema-gillaspy} says that $C^*\sr(\Lambda) \sim_{ME} \Z^k \ltimes C^*\sr(\Lambda \times_d \Z^k)$. Hence, understanding how group homomorphisms affect the (spectral sequence converging to the) $\Kk$-theory of a crossed product by $\Z^k$ will enable us to analyze the relationship between the $\Kk$-theory of real directed graph $\cs$-algebras, and the $\Kk$-theory of real $k$-graph $\cs$-algebras.  We will detail this relationship in future work \cite{BBG-embeddings}.

We also hope that by concisely articulating the implications of \cite{miller-hlogy} for the $\Kk$-theory of (real and complex) crossed-product $\cs$-algebras by discrete groups, as we do in Theorem \ref{KO-sp.seq.summary} and Remark \ref{rmk:main-for-complex}, we will facilitate future applications of \cite[Theorem 6.1]{miller-hlogy} in that context. The relationship between homology and $\cs$-algebraic $\Kk$-theory for groupoids and crossed products has drawn an increasing amount of attention in recent years  (for a very incomplete list, cf.~\cite{matui-HK, scarparo, bdgw,deeley, proietti-yamashita-iv, li-2025}).  Since \cite[Theorem 6.1]{miller-hlogy} establishes the functoriality of this relationship with respect to group homomorphisms,\footnote{and a more general class of groupoid morphisms} 
we anticipate that this work will facilitate  $\Kk$-theory computations in previously inaccessible settings.

This paper is structured as follows.  The 
Preliminaries section includes Proposition \ref{prop:ABC-spectral}, which establishes the existence of an ABC morphism in  our case of interest, and some classical but not universally known results about the computation of group homology from acyclic resolutions. Section \ref{sec:hlogy} relates the derived functors of the ABC spectral sequences to group homology, taking inspiration from \cite[Remark 4.6 and Theorem 5.14]{miller-hlogy}, and culminating in our main technical theorem Theorem \ref{group homology theorem}. This theorem establishes that the $E^2$-page map induced by a group homomorphism $\Omega\colon G \to H$, which is defined as a map between derived functors,  is indeed given by the natural homomorphism of homology groups.
This leads (in our main results, Theorems \ref{KO-sp.seq.summary} and \ref{CRT-sp.seq.summary}) to our desired, self-contained statement of the implications of \cite[Theorem 6.1]{miller-hlogy} for discrete crossed products of real $\cs$-algebras.

\section*{Acknowledgments}
This research was initiated during a visit of the third-named author to the University of Southern Denmark (SDU), which was facilitated by a Faculty International Activity Award from the University of Montana. E.G.~thanks the Operator Algebras group at SDU for the invitation and their hospitality.

A.M. was supported by the Independent Research Fund Denmark through Grant 1054-00094B and by the Research Foundation Flanders (FWO) through Project 1212126N.

The authors thank Jamie Gabe for helpful conversations, and in particular for suggesting the current collaboration.

\section{Preliminaries}

Our  goal is Theorem \ref{CRT-sp.seq.summary}, which establishes the naturality in $G$ of a spectral sequence which converges to the united $\Kk$-theory (see Section \ref{subsection-unitedK}) of the crossed product $G \ltimes_r A$ of a real $\cs$-algebra $A$.  This spectral sequence arises from an ABC tuple \cite[Definition 2.11]{miller-hlogy}, which is a construct from triangulated categories. We summarize the relevant features of these objects here, but  in the interests of brevity, we refer the reader to the literature \cite{meyer-nest,  meyer-nest-2010, meyer-2008, bonicke-proietti, miller-hlogy} for detailed definitions.

At several points in this paper, we restrict our attention to countable discrete groups $G$ which satisfy the {\em Baum--Connes conjecture} with coefficients in a separable real $G$-$\cs$-algebra $A$.   This conjecture (in the real case)
 {asserts that the {\em real assembly map} 
 \[ \mu_i(G,A) \colon \KO_i^{G}(\underline{E} G;A) \to \KO_i(G \ltimes_{r} A)\]
is an isomorphism 
\cite{schick, baum-karoubi, rosenberg2016}.
It is known that for a given $G$ and $A$, the real assembly map $\mu_i(G,A)$ is an isomorphism for all $i$ if and only if the complex assembly map 
$\mu\pc_i(G,A\sc)$ for the complexification $A_{\mathbb C}$ is an isomorphism for all $i$ (see Theorem~2.10 in \cite{schick} and the main theorem of \cite{baum-karoubi}).
When this happens the $\cs$-algebraic $\Kk$-theory $\KO_*(G \ltimes_r A)$ and $\Kk_*(G \ltimes_r A\sc)$ become accessible via topological tools (e.g.~classifying spaces and localizations). 
For any amenable group $G$ and any $G$-$\cs$-algebra $A$, for example, both the real and complex assembly maps are isomorphisms.
 For further information in the complex setting, see the Baum--Connes extended survey \cite{BC-survey}.

In this paper, we exclusively invoke the Baum--Connes conjecture in the proof of Proposition \ref{prop:ABC-spectral}, where we use the formulation of the assembly map from \cite[Theorem 5.2]{meyer-nest} in terms of localizations. }

\subsection{Crossed products and homomorphisms}

Let $A$ be a real or complex $\cs$-algebra equipped with an action of a discrete group $H$. The (full) \emph{crossed product} $H \ltimes A$ is the universal (real or complex) $\cs$-algebra generated by $A$ and the full group $\cs$-algebra $\cs(H) = \cs( u_h \mid h \in H )$ such that $u_h a u_h^* = h \cdot a$ for each $h \in H$ and $a \in A$. A group homomorphism $\Omega \colon G \to H$ equips $A$ with an action of $G$ and induces a $*$-homomorphism \[\Omega \ltimes \id_A \colon G \ltimes A \to H \ltimes A\] satisfying $(\Omega \ltimes \id_A)(a) = a$ and $(\Omega \ltimes \id_A)(u_g) = u_{\Omega(g)}$.

Given a (real or complex) representation $\rho \colon A \to \mathcal B(\mathcal H_\rho)$ of $A$, the \emph{regular representation} $\lambda^H_\rho$ of $H \ltimes A$ on the Hilbert space $\ell^2(H,\mathcal H_\rho)$ is given, for $a \in A$ and $h \in H$, by 
\begin{align*}
(\lambda^H_\rho(a) \xi)(s) & = \rho(s^{-1} \cdot a) \xi(s) & (\lambda^H_\rho (u_h)\xi)(s) & = \xi(h^{-1}s)
\end{align*}
for $s \in H$ and $\xi \in \ell^2(H,\mathcal H_\rho)$. The \emph{reduced crossed product} $H \ltimes_r A$ is the completion of $H \ltimes A$ with respect to the seminorm 
\[ \| \xi \|_r = \sup \{ \|\lambda^H_\rho(\xi)\| \mid \rho \colon A \to \mathcal B(\mathcal H_\rho) \text{ a representation of } A  \}, \, \xi \in H \ltimes A. \]
We note that if $\rho$ is a faithful representation of $A$, then $\| \xi \|_r = \|\lambda^H_\rho(\xi)\|$ for each $\xi \in H \ltimes A$. 
See \cite[Section~1.3]{schroderbook} and \cite[Section 2]{boersema-rmj} for more on crossed products for real $\cs$-algebras.

Suppose that $A$ is a real $H$-$\cs$-algebra with complexification $A_{\mathbb C}$. Then $H \ltimes A_{\mathbb C}$ may be identified with the complexification of $H \ltimes A$ via the natural inclusion $H \ltimes A \to H \ltimes A_{\mathbb C}$. Moreover, the reduced crossed product $H \ltimes_r A_{\mathbb C}$ is the complexification of $H \ltimes_r A$. This implies that the reduced seminorm on $H \ltimes A$ is inherited from the reduced seminorm on $H \ltimes A_{\mathbb C}$.

In order to obtain our main results (Theorem \ref{KO-sp.seq.summary} and Theorem \ref{CRT-sp.seq.summary}) we will need the following result about crossed products. Even for complex $\cs$-algebras, while this result is undoubtedly known to experts, we could not locate a proof in the literature. Hence we include a proof here in both the real and complex cases. {The argument is heavily based on the book chapter \cite[\S 2]{CELY}.}

\begin{prop}
Let $\Omega \colon G \to H$ be a homomorphism of discrete groups with amenable kernel, and let $A$ be a (real or complex) $H$-$\cs$-algebra. Then the induced $*$-homomorphism $\Omega \ltimes \id_A \colon G \ltimes A \to H \ltimes A$ descends to a $*$-homomorphism $G \ltimes_r A \to H \ltimes_r A$ (which we denote $\Omega \ltimes_r \id_A$).
\label{prop:amenable-kernel}
\end{prop}
\begin{proof}
We first prove this statement for a complex $\cs$-algebra $A$.
As $\Omega(G) \ltimes_r A$ {embeds}
in $H \ltimes_r A$ we may further assume that $\Omega \colon G \to H$ is surjective. 
Let $N = \ker \Omega$ and let $\rho \colon A \to \mathcal B(\mathcal H_\rho)$ be a faithful representation of $A$, so that 
$\lambda^H_\rho(H \ltimes A) = H\ltimes_r A$ and $\lambda^G_\rho(G \ltimes A) = G \ltimes_r A$.  
Our goal is to show that the regular representation $\lambda^G_\rho$ of $G \ltimes A$ weakly contains the representation $(\Omega \ltimes \id_A)^*(\lambda^H_\rho)$ on $\ell^2(H,\mathcal H_\rho)$.

We shall utilise the subgroup induction functor $\Ind^G_N \colon \mathrm{Rep}(N \ltimes A) \to \mathrm{Rep}(G \ltimes A)$ \cite[Definition 2.7.1]{CELY} and additionally $\Ind^G_{\{e\}}$ and $\Ind^N_{\{e\}}$. 
By \cite[Remark 2.7.5(3)]{CELY}, the regular representation $\lambda^G_\rho$ of $G \ltimes A$ on $\ell^2(G,\mathcal H_\rho)$ is equivalent to $\Ind^G_{\{e\}} \rho$. Thus, by induction-in-stages \cite[Theorem 2.7.10]{CELY}, $\lambda^G_\rho \sim \Ind^G_{\{e\}} \rho$ is further equivalent to $\Ind^G_N \Ind^N_{\{e\}} \rho \sim \Ind^G_N \lambda^N_\rho$. By amenability of $N$, $\lambda^N_\rho$ weakly contains the trivial representation $\mathrm{tr} \times \rho$ of $N \ltimes A$ on $\mathcal H_\rho$ \cite[Proposition 2.4.5]{CELY}. Continuity of induction \cite[Proposition 2.7.4]{CELY} implies a weak containment $\Ind^G_N (\mathrm{tr} \times \rho) \prec \Ind^G_N \lambda^N_\rho \sim \lambda^G_\rho$. 

By \cite[Corollary 2.7.8]{CELY}, the representation $\Ind^G_N(\mathrm{tr} \times \rho)$ of $G \ltimes A$ is equivalent to the representation $\pi$ of $G \ltimes A$ on the Hilbert space $\ell^2(G/N,\mathcal H_\rho)$ given by $(\pi(u_t)\xi)(gN) = \xi(t^{-1}gN)$ and $(\pi(a)\xi)(gN) = \rho(g^{-1} \cdot a)\xi(gN)$ for $a \in A$, $g,t \in G$ and $\xi \in \ell^2(G/N,\mathcal H_\rho)$. Identifying $G/N$ with $H$, this is simply the regular representation $\lambda^H_\rho$ of $H \ltimes A$ composed with the $*$-homomorphism $\Omega \ltimes \id_A \colon G \ltimes A \to H \ltimes A$. Putting everything together, this representation $(\Omega \ltimes \id_A)^* (\lambda^H_\rho)$ is weakly contained in the regular representation $\lambda^G_\rho$, and so we obtain a $*$-homomorphism $G \ltimes_r A \to H \ltimes_r A$.

{Now, suppose that $A$ is a real $H$-$\cs$-algebra with complexification $A_{\mathbb C}$ and let $\eta \in G \ltimes A$. Using that the reduced norm $\|\cdot\|_r$ on   $G \ltimes_r A$ is inherited from the reduced norm $\| \cdot \|_{r,\mathbb C}$ on $G \ltimes_r A_{\mathbb C}$ and similarly for $H$, we compute 
\[
\|\Omega \ltimes \id_A (\eta) \|_r  = \|  \Omega \ltimes \id_A(\eta) \|_{r,\mathbb C} = \|\Omega \ltimes \id_{A_{\mathbb C}}(\eta) \|_{r,\mathbb C} \leq \| \eta \|_{r,\mathbb C} = \| \eta \|_r.
\]
That is, $\Omega \ltimes \id_A\colon G \ltimes_r A \to H \ltimes_r A$ is indeed a $*$-homomorphism, as claimed.}
\end{proof}

\subsection{KK-theory and triangulated categories}

We denote by $\KK$ the category of real separable $\cs$-algebras with morphisms the elements of \textcolor{black}{the real KK-groups} $\KK_*(A,B)$.  Similarly, for a (countable discrete) group $G$, $\KK^G$ denotes the category whose objects are separable real $\cs$-algebras with an action of $G$, and whose morphisms are classes in $\KK^G(A,B)$. These categories are examples of {\em triangulated categories} as was verified in \cite{meyer-nest},
in the same way that the analogous categories of complex $\cs$-algebras, or complex $\cs$-algebras with $G$-action, are.

Consider the (triangulated) functors 
$$\Res_G\colon \KK^G \to \KK \quad \text{and} \quad \Ind_G\colon \KK \to \KK^G ,$$ 
where $\Res_G$ forgets the $G$-action, and $\Ind_G A = \bigoplus_{g\in G} A = C_0(G, A)$, where the $G$-action is given by permuting coordinates: if $f \in C_0(G, A),$ $g \cdot f(h) := f(g^{-1}h)$. Then $\mathfrak I^G := \ker \Res_G$ is a homological ideal in the sense of \cite[Definition 2.20]{meyer-nest}. 
(In this paper, our interaction with the analogous homological ideal in complex $\KK$-theory will be minimal, so we will refrain from decorating $\mathfrak I^G$ with its coefficient field.)

As we are working with real $\cs$-algebras, our initial focus is the 8-fold periodic real $\Kk$-theory functor, $\KO_*$. As $\KO_*$ takes values in an Abelian category, it is easily checked that 
the functor $\KO_*( G\ltimes \Input )\colon \KK^G\to {\bf Ab}$ is a stable homological functor in the sense of \cite[Definition 2.14]{meyer-nest-2010}; that is, 
$$\KO_*( G\ltimes (C_0\pr(\R)\otimes A)) \cong \Sigma \KO_*( G\ltimes A) \; $$
where $\Sigma$ is the shift functor applied to a graded abelian group.
Also, since $\Res_G $ and $\Ind_G$  are adjoint functors (see \cite[Section 3.2]{meyer-nest}; c.f.~also \cite[Theorem 6.2]{bonicke-advances}),  $\mathfrak I^G = \ker \Res_G $ has enough projective objects.
Moreover, $\KK^G$ has countable direct sums which are compatible with $\Res_G$ and thus $\mathfrak I_G$ and also with $\KO_*(G \ltimes \Input)$.
Consequently, for any $\cs$-algebra $A \in \KK^G$, the quadruple $(\KK^G,\mathfrak I^G, \KO_*(G \ltimes \Input), A)$ is an {\em ABC tuple} \cite[Definition 2.11]{miller-hlogy}. 

From an ABC tuple, we obtain a {convergent} spectral sequence (called the {\em ABC spectral sequence}) by \cite[Theorem 5.1]{meyer-2008}.  In our case of interest, this translates into the following result. {In Section \ref{sec:hlogy} below, we will revisit the derived functors $\mathbb L_p^{\mathfrak I^G}(\KO_q(G \ltimes A))$ introduced in the proof below, and establish an alternative description of them using group homology.}

 \begin{prop}
\label{prop:ABC-spectral}
	Let $G$ be a countable discrete group. The ABC spectral sequence of the ABC tuple $$(\KK^G, \mathfrak I^G, \KO_*(G \ltimes \underline{\quad}),A )$$ has $E^2$ page given by the derived functors 
    $$ E^2_{p,q} = \mathbb L_p^{\mathfrak I^G}(\KO_q(G \ltimes A)) $$
 with respect to $\mathfrak I^G$. Moreover, if $G$  is torsion-free and satisfies the Baum--Connes conjecture with coefficients in a real $\cs$-algebra $A$, then the ABC spectral sequence converges to $\KO_*(G \ltimes_r A)$. Furthermore, the spectral sequence is natural in $A$: morphisms $\psi \in \KK^G(A,B)$ functorially induce morphisms $\psi$ of spectral sequences such that the induced maps $\psi_n \colon  \mathbb L^{\mathfrak I^G}_n(\KO_*(G \ltimes A)) \to \mathbb L^{\mathfrak I^G}_n(\KO_*(G \ltimes B))$ converge to $\psi_*  \colon \KO_*(G \ltimes_r A) \to \KO_*(G \ltimes_r B)$.
\end{prop}

\begin{rmk}
\label{rmk:full-vs-reduced}
The use of both full and reduced crossed products in the statement above (and, indeed, throughout the paper) is deliberate. 
We use the full crossed product in the ABC tuple $(\KK^G, \mathfrak I^G, \KO_*(G \ltimes \underline{\quad}),A )$ for its greater functoriality properties. 
A reduced version of this ABC tuple is used in \cite{meyer-nest} to model the Baum--Connes assembly map, with its spectral sequence converging to the topological K-theory. However, as we shall see in the following proof, the spectral sequence for the full ABC tuple also converges to the topological K-theory. Thus we use the full crossed product for the spectral sequence for better functoriality and we also use the reduced crossed product so that, after applying the Baum--Connes assembly map, we 
obtain a spectral sequence converging to the $\Kk$-theory of our crossed product.
\end{rmk}

\begin{proof}
	Theorem 5.1 of \cite{meyer-2008}
    shows that the ABC spectral sequence given by this initial data converges to the localization $\mathbb L \KO_*(G \ltimes A)$. 
   Naturality in $A$ may be deduced from \cite[Theorem 4.8]{meyer-2008} and the ensuing paragraph, and follows explicitly from \cite[Theorem 6.1]{miller-hlogy}.
     
    The natural transformation $\KO_*(G \ltimes \underline{\quad}) \Rightarrow \KO_*(G \ltimes_r \underline{\quad})$ induces a map $\kappa \colon \mathbb L \KO_*(G \ltimes A) \to \mathbb L \KO_*(G \ltimes_r A)$  natural in $A$ (see \cite[Definition 5.10]{miller-hlogy}) given by applying the natural transformation to the localization $LA \in \KK^G$. This $\cs$-algebra lives in the localizing subcategory generated by induced $G$-$\cs$-algebras $\Ind_G B$,\footnote{The localization $LA$ is written $\widetilde A$ and called a $\mathscr{CI}$-simplicial approximation of $A$ in \cite[\S 4]{meyer-nest}.} 
     which are proper $G$-algebras, on which %
     the full and reduced crossed products agree. 
    It follows that $\kappa$ is a natural isomorphism (cf.~\cite[Remark 2.6]{miller-hlogy}). 
    
By \cite[Theorem 5.2]{meyer-nest}, if $G$ is torsion-free, the Baum--Connes assembly map is naturally isomorphic to the localization map $\mu \colon \mathbb L \KO_*(G \ltimes_r A) \to \KO_*(G \ltimes_r A)$ for the reduced ABC tuple $(\KK^G, \mathfrak I^G, \KO_*(G \ltimes_r \underline{\quad}), A)$. If $G$ additionally satisfies the Baum--Connes conjecture with coefficients $A$, $\mu$ is an isomorphism. Composing $\kappa$ with $\mu$, the spectral sequence converges to $\KO_*(G \ltimes_r A)$. 

The description of the $E^2$ page follows from \cite[Theorem 4.8]{meyer-2008}.
\end{proof}

The analogous result for complex $\Kk$-theory is also true; see \cite[Examples 2.12 and 2.20]{miller-hlogy}. 

\begin{rmk}
\label{rmk:torsion-free}
The assumption that $G$ is torsion-free in Proposition \ref{prop:ABC-spectral} ensures that the homological ideal $\mathfrak I^G = \ker \Res_G$ is the right one to induce the Baum--Connes assembly map in the machinery of \cite{meyer-nest}. In general, for a countable discrete group $G$, to obtain the Baum--Connes assembly map, one should take the homological ideal 
\[ \mathfrak I^G_{\mathrm{fin}} = \bigcap_{F \leq G \, \text{finite}} \ker \Res^F_G,\]
where (for a finite subgroup $F$ of $G$) $\Res^F_G \colon \KK^G \to \KK^F$ is the (triangulated) functor which remembers only the $F$-action.
However, the resulting derived functors $\mathbb L_p^{\mathfrak I^G_{\mathrm{fin}}}(\KO_q(G \ltimes A))$ occupying the $E^2$ page become more complicated to describe.\footnote{In light of \cite{mislin-book} these should resemble Bredon homology groups.} In particular, if $G$ has torsion so that (a priori) $\mathfrak I^G_{\mathrm{fin}} \ne \mathfrak I^G$, then we cannot apply our Proposition \ref{prop:derived-is-homology} below, which gives an isomorphism $\mathbb L_p^{\mathfrak I^G}(\KO_q(G \ltimes A)) \cong \Hh_p(G;\KO_q(A))$. 
\end{rmk}

\subsection{Spectral sequences}
\label{sec:}

{
For a reader already familiar with spectral sequences, we review some of the definitions and terminology here. For more details see \cite[Chapter 5]{Weibel} and \cite{Boardman}.

{Let $\{E^r_{p,q}, d^r\}_{r \geq 2}$ be a spectral sequence starting at the second page. Recall that the differential $d^r$ has degree $(-r, r-1)$, and the data of a spectral sequence comes equipped with specified identifications $E^{r+1} \cong \Hh(E^r,d^r)$. }

The limit sheet $E^\infty_{p,q}$ is defined through the auxiliary groups $Z^r_{p,q}$ and $B^r_{p,q}$ of cycles and boundaries, which are defined inductively so that 
\[ 0 = B^2_{p,q} \subseteq B^3_{p,q} \subseteq \cdots \subseteq B^r_{p,q} \subseteq \cdots \subseteq Z^r_{p,q} \subseteq \cdots \subseteq Z^3_{p,q} \subseteq Z^2_{p,q} = E^2_{p,q} \]
with canonical isomorphisms $Z^r_{p,q}/B^r_{p,q} \cong E^r_{p,q}$. Then $$E^\infty_{p,q} := \bigcap_{r \geq 2} Z^r_{p,q} / \bigcup_{r \geq 2} B^r_{p,q}.$$
Note that if the spectral sequence stabilizes in the sense that the differentials $d^r $ vanish for all $r \geq r_0$, then $E^\infty_{p,q} = E^{r_0}_{p,q}.$

When we say that a spectral sequence converges to a graded module $(A_n)_{n \in \mathbb Z}$, we mean in the strong sense of \cite{Boardman}, i.e.~that there is a specified complete Hausdorff exhaustive filtration $(F_p A_n)_{p \in \mathbb Z}$ of $A_n$ for each $n$,
\begin{equation}\label{filtration}
\cdots \subseteq F_{p-1} A_n \subseteq F_p A_n 
\subseteq  F_{p+1} A_n 
\subseteq \cdots \subseteq A_n, 
\end{equation}
with specified isomorphisms
\begin{equation}\label{filtrationisomorphism}
F_{p} A_n / F_{p-1} A_n \cong E^\infty_{p,n-p}
\end{equation}
for each $p, n \in \mathbb Z$. Exhaustive means that $A_n = \bigcup_{p \in \mathbb Z} F_p A_n$. In our setting we will always have $F_0 A_n = 0$, which automatically implies that the filtration $(F_p A_n)_{p \in \mathbb Z}$ is complete and Hausdorff.
We denote this convergence by 
$$E^r_{p,q} \Longrightarrow A_{p+q}.$$

In this paper, we are concerned with naturality of our spectral sequences, so we will encounter morphisms of spectral sequences. 

A morphism of spectral sequences of degree $(i,j)$ is a collection of maps
$f^r_{p,q} \colon E^{r}_{p,q} \rightarrow E^{'r}_{p+i,q+j}$ that intertwines the differential maps $d^r$ and $d^{'r}$:
$$ d^{'r} \circ f^r_{p,q} = f^r_{p-r,q+r-1} \circ d^r \; $$
and such that the map $f^{r+1}_{p,q}$ is the map induced on homology by $f^{r}_{p,q}$. If not otherwise specified, a morphism of spectral sequences should be taken to have degree $(0,0)$.

A 
morphism of spectral sequences  is said to converge to a graded homomorphism $\alpha_* \colon A_* \rightarrow A'_*$ if the following hold:
\begin{enumerate}[(1)]
\item $E^r_{p,q}$ converges to $A_n$ and $E^{'r}_{p,q}$ converges to $A'_n$. In particular, this means that there are filtrations on $A$ and $A'$ as in \eqref{filtration} and isomorphisms as in \eqref{filtrationisomorphism}.
\item
The map $\alpha_n \colon A_n \rightarrow A'_n$ respects the filtrations in the sense that 
$$\alpha_n(F_p A_n) \subseteq F_p A'_n$$
for all $p$ and $n$.
\item The map induced by $\alpha_n$ on the subquotients 
$$F_p A_n/F_{p-1} A_n   \rightarrow
    F_p A'_n/F_{p-1} A'_n$$
corresponds to 
$$f^\infty_{p,q} \colon E^{\infty}_{p,q} \rightarrow E^{' \infty}_{p,q} $$
via the isomorphisms of (\ref{filtrationisomorphism}).
\end{enumerate}
(Compare to Definition~5.2.1, Lemma~5.2.4, and Theorem~5.2.12 in \cite{Weibel}.)

In Section~\ref{subsection-unitedK} we will introduce united $\Kk$-theory, which is a triple of graded abelian groups, representing three variants of $\Kk$-theory.
There we obtain a system of three spectral sequences that converges to the united $\Kk$-theory of a real crossed product (see 
Theorem~\ref{CRT-sp.seq.summary}).
}

\subsection{Group Homology}
\label{sec:hlogy-via-acyclic}
In this section, we will review some salient facts about group homology that will be used for our work. Given a discrete group $G$ and a $G$-module $M$, let $M_G$ be the coinvariants module, consisting of the $G$-orbits of $M$. That is,
$$M_G = M/K$$
where $K$ is the submodule of $M$ generated by elements of the form $m - g \cdot m$ for $m \in M$ and $g \in G$.

\begin{defn}[Definition~6.1.2 of \cite{Weibel}]
Let $G$ be a group and let $M$ be a $G$-module. Then the group homology $\Hh_*(G;M)$ is defined by
$$\Hh_*(G;M) = \Hh_*((Q_\bullet)_G) \; $$
where $Q_\bullet \to M$ is any projective resolution of $G$-modules.
\end{defn}

This does not depend on the choice of projective resolution by the Fundamental Lemma of Homological Algebra (cf.~\cite[Lemma 2.4.1]{Weibel}). 

In fact, the resolution used to compute $\Hh_*(G;M)$ does not have to be projective -- 
we can instead use any $G$-acyclic resolution,
thanks to the following Proposition. {(A $G$-module $Q$ is called \emph{$G$-acyclic} if $\Hh_i(G;Q) = 0$ for each $i \geq 1$.)} We will make use of this result in Proposition~\ref{prop:derived-is-homology} below.

\begin{prop}[cf. Remark~2.4.3 of \cite{Weibel}] \label{prop-acyclic-res}
Let $G$ be a group and let $M$ be a $G$-module. Then the group homology can be computed as
$$\Hh_*(G;M) \cong \Hh_*((Q'_\bullet)_G) \; $$
where $Q'_\bullet \to M$ is any $G$-acyclic resolution of $G$-modules. 
\end{prop}

We refer the reader to Remark~2.4.3 of \cite{Weibel} (as well as the linked exercise) for the proof of this proposition, but the isomorphism can be described as follows. Suppose that $Q_\bullet \to M$ is a projective resolution and $Q_\bullet' \to M$ is a $G$-acyclic resolution. Then there is a unique (up to chain homotopy) chain map $Q_\bullet \to Q_\bullet'$ over $M$ and this chain map then induces an isomorphism in homology $\Hh_n((Q_\bullet)_G) \to \Hh_n((Q_\bullet')_G)$. 

Now, we turn to a discussion of functoriality of group homology with respect to homomorphisms. If $\Omega \colon G \to H$ is a homomorphism of groups, then an $H$-action on a group $M$ pulls back to a $G$-action on $M$ by
$$g \cdot m = \Omega(g) \cdot m\; .$$ 
The resulting $G$-module is denoted by $\Res^G_H M$. 

\begin{prop} \label{prop-grouphomology-natural}
Let $\Omega \colon G \to H$ be a homomorphism of discrete groups. Then there is a natural induced function
$$\Hh_n(\Omega; M) \colon \Hh_n(G; \Res^G_H M) \to \Hh_n(H; M). $$
\end{prop}

We sketch here the construction of $\Hh_n(\Omega;M)$.
Let $P_\bullet \to \Res^G_H M$ be any projective resolution (of $G$-modules) and $Q_\bullet \to M$ be any projective resolution (of $H$-modules), used to define the corresponding homology groups. By applying the restriction functor to the latter resolution, we obtain a new resolution of $G$-modules,
$$\Res^G_H Q_\bullet \to \Res^G_H M  ,$$
and by projectivity there exists a chain map $f_\bullet$ of resolutions
\[\begin{tikzcd}
	{P_\bullet} & {\Res^G_H M} \\
    {\Res^G_H Q_\bullet} & {\Res^G_H M}. 
	\arrow["f_\bullet", from=1-1, to=2-1]
    \arrow[from=1-1, to=1-2]
    \arrow[ from=2-1, to=2-2]
    \arrow["=", from=1-2, to=2-2]
\end{tikzcd}\] 

For any $H$-module $N$ the canonical and natural map $\Res^G_H N \to N_H$ is trivially $G$-invariant, giving us a natural map $\mu_N \colon (\Res^G_H N)_G \to N_H$. In particular, applying this construction to $Q_\bullet$ and $M$ we obtain 
the vertical chain maps in the commutative diagram below.
\[\begin{tikzcd}
    {(\Res^G_H Q_\bullet)_G} & {(\Res^G_H M)_G} \\
    {(Q_\bullet)_H} & {M_H} 
	\arrow["\mu_{Q_\bullet}", from=1-1, to=2-1]
    \arrow[from=1-1, to=1-2]
    \arrow[ from=2-1, to=2-2]
    \arrow["\mu_M", from=1-2, to=2-2]
\end{tikzcd}\]
Then the map $\Hh_n(\Omega;M)$ is induced on homology by the chain map 
\begin{equation} \mu_{Q_\bullet} \circ (f_\bullet)_G \colon (P_\bullet)_G \to (Q_\bullet)_H  .
	\label{eq:Hn-Omega}
	\end{equation} 

We also point out that in this construction we can take $P_\bullet$ to be a $G$-acyclic resolution and we can take $Q_\bullet$ to be an $H$-acyclic resolution, provided 
that the chain map $f_\bullet\colon P_\bullet \to Q_\bullet$ can be constructed, which may not always be the case for $G$-acyclic resolutions.

Indeed, suppose that $P'_\bullet \to \Res^G_H M$ is a $G$-acyclic resolution, $Q'_\bullet \to M$ is an $H$-acyclic resolution and $f'_\bullet \colon P'_\bullet \to \Res^G_H Q'_\bullet$ is a chain map of resolutions. By the Fundamental Lemma of Homological Algebra, the unique (up to chain homotopy) chain maps $P_\bullet \to P'_\bullet$ and $Q_\bullet \to Q'_\bullet$ of resolutions induce a diagram
\[\begin{tikzcd}
	{(P_\bullet)_G} & {(P'_\bullet)_G} \\
	{(\Res^G_H Q_\bullet)_G} & {(\Res^G_H Q'_\bullet)_G} \\
	{(Q_\bullet)_H} & {(Q'_\bullet)_H}
	\arrow[from=1-1, to=1-2]
	\arrow["{f_\bullet}", from=1-1, to=2-1]
	\arrow["{f'_\bullet}", from=1-2, to=2-2]
	\arrow[from=2-1, to=2-2]
	\arrow["{\mu_{Q_\bullet}}", from=2-1, to=3-1]
	\arrow["{\mu_{Q'_\bullet}}", from=2-2, to=3-2]
	\arrow[from=3-1, to=3-2]
\end{tikzcd}\] 
which commutes up to chain homotopy. Thus, using $\Hh_n((P'_\bullet)_G)$ and $\Hh_n((Q'_\bullet)_H)$ to compute $\Hh_n(G;\Res^G_H M)$ and $\Hh_n(H;M)$ respectively, the map $\Hh_n(\Omega;M)$ is induced on homology by the chain map 
$$\mu_{Q'_\bullet} \circ (f'_\bullet)_G \colon (P'_\bullet)_G \to (Q'_\bullet)_H .$$

\section{ABC derived functors as homology}\label{sec:hlogy}

Let $G$ be a countable discrete group and let $A$ be a separable real $G$-$\cs$-algebra. For this section we will not need any further assumptions on $G$.

In Proposition \ref{prop:ABC-spectral},  certain derived functors $\mathbb L^{\mathfrak I^G}_n\KO_i(G \ltimes  A)$  show up on the $E^2$ page of a spectral sequence, which converges to the operator K-theory $\KO_*(G \ltimes_r A)$, when $G$ is torsion-free and satisfies the Baum--Connes conjecture with coefficients in $A$. We show in Proposition \ref{prop:derived-is-homology} below that, even without these additional assumptions, these derived functors are simply given by group homology groups, and that moreover (Theorem \ref{group homology theorem}) this identification is functorial with respect to group homomorphisms.

The derived functors $\mathbb L_n^{\mathfrak I^G} \KO_i(G \ltimes  \Input)$ are defined as follows on $\KK^G$. Let $A \in \KK^G$ be a separable real $G$-$\cs$-algebra and let $P_\bullet \to A$ be an $\mathfrak I^G$-projective resolution of $A$ in $\KK^G$. This means we have a chain complex 
$$ \cdots \to P_n \to \cdots \to P_0 \to A \to 0$$ 
which is $\mathfrak I^G$-exact in the sense of \cite[Definition 3.7]{meyer-nest-2010} with each $P_n$ an $\mathfrak I^G$-projective object in the sense of \cite[Definition 3.21]{meyer-nest-2010}. The $n$th derived functor $\mathbb L_n^{\mathfrak I^G} \KO_i(G \ltimes A)$ is given by the $n$th homology of the induced chain complex 
\[ \cdots \to \KO_i(G \ltimes P_n) \to \cdots \to \KO_i(G \ltimes P_0) \to 0. \]
A different choice of $\mathfrak I^G$-projective resolution produces a canonically isomorphic sequence of homology groups by \cite[Proposition 3.26]{meyer-nest-2010}.  
The same proposition is also used to define functoriality: given a morphism $f \colon A \to B$ in $\KK^G$ and $\mathfrak I^G$-projective resolutions $P_\bullet \to A$ and $Q_\bullet \to B$, \cite[Proposition 3.26]{meyer-nest-2010} guarantees the existence of a chain map $P_\bullet \to Q_\bullet$ over $f$, which is used to produce a homomorphism $\mathbb L_n^{\mathfrak I^G} \KO_i(G \ltimes f) \colon \mathbb L_n^{\mathfrak I^G} \KO_i(G \ltimes A) \to \mathbb L_n^{\mathfrak I^G} \KO_i(G \ltimes B)$.

We will make use of an explicit $\mathfrak I^G$-projective resolution $P_\bullet \to A$ constructed from the induction-restriction adjunction $\Ind_G \dashv \Res_G$. For $A \in \KK^G$ we have $\Res_G A = A \in \KK$ (forgetting the $G$-structure).  
For a separable real $\cs$-algebra $B \in \KK$ (with no $G$-action), we have 
\begin{equation}
 \Ind_G(B) = C_0(G, B) = C_0\pr(G) \otimes B\label{eq:Ind-G-orig} 
\end{equation}
where the $G$-action $\alpha \colon G \rightarrow \Aut(\Ind_G B)$ is defined by 
$$(\alpha(g) ( f))(h) = f(g^{-1}h)$$ 
for $g \in G$ and $f \in C_0(G, B)$.
When $B = \Res_G(A)$, \eqref{eq:Ind-G-orig} becomes 
\begin{equation*}\Ind_G \Res_G A = C_0(G, \Res_G(A)) = C_0\pr(G) \otimes \Res_G(A).
\end{equation*}

As discussed in \cite[Section 3.2]{meyer-nest} (cf.~also \cite[Theorem 6.2]{bonicke-advances} in the setting where $H = \{e\}$), 
$\Ind_G$ and $\Res_G$ are adjoint functors.  Consequently, for any $A \in \KK^G$, a standard construction in homological algebra (\cite[Section 8.6]{Weibel} or \cite[Section 3.1]{proietti-yamashita-2022}) yields an $\frak I^G$-projective resolution $P_\bullet \to A$:  we have 
\[P_n = (\Ind_G\Res_G)^{n+1}(A),\]
and the morphism $\delta_n\colon P_n \to P_{n-1}$ for $n \geq 0$ (with the convention that $P_{-1} = A$) is an alternating sum of the ``face maps''  as described in the proof of \cite[Theorem 3.1]{proietti-yamashita-2022}. The face maps are natural with respect to morphisms $A \to B$ in $\KK^G$.

{We have one more ingredient. Given a separable real $G$-$\cs$-algebra $A \in \KK^G$ the inclusion $A \to G \ltimes A$ induces a map $\KO_i(A) \to \KO_i(G \ltimes A)$ in K-theory which is $G$-invariant by invariance of K-theory under unitary conjugation. It therefore descends to a homomorphism $\nu_A \colon \KO_i(A)_G \to \KO_i(G \ltimes A)$ which is natural with respect to morphisms in $\KK^G$.} 

The following lemma will be needed in the proof of Proposition \ref{prop:derived-is-homology} below.  This result is unsurprising and may be known to experts, but we include it here for completeness.
Note that for any discrete group $G$ we have $G \ltimes_{\rm{lt}} C_0\pr(G) \cong \mathcal \K(\ell^2\sr(G)) \cong G \ltimes_{\rm{lt}, r} C_0\pr(G)$, so we refrain from distinguishing between  the full and crossed products here.  

\begin{lemma}
	\label{lem:takai}
	For any real $\cs$-algebra $A$ the inclusion
	$$\iota_A \colon A \hookrightarrow 
	G \ltimes_{\rm{lt}} \Ind_G(A) $$
    given by the composition of the inclusions $A \hookrightarrow \Ind_G(A)$ (by $a \mapsto \delta_e \otimes a$) and $\Ind_G(A) \hookrightarrow G \ltimes_{\rm{lt}, r} \Ind_G(A)$
	is an isomorphism on $\Kk$-theory.
\end{lemma}

\begin{proof} 
Observe that $\iota_A$ is equivalent to the factorization 
$$ \iota_\R \otimes \id_A \colon A \rightarrow (G \ltimes_{\rm{lt}} C_0\pr(G)) \otimes A ,$$ in light of Equation \eqref{eq:Ind-G-orig}. From \cite[Theorem 7]{boersema-rmj} (see also Lemma~7.5 of \cite{williamsbook} in the complex case) there is an isomorphism 
$\Theta \colon G \ltimes_{\text{lt}} C_0\pr(G)
	\rightarrow \K(\ell\sr^2(G))  $	
given by
	$$\Theta(f)(h)(r) = \sum_{s \in G} f(s)(r) h(s^{-1} r). $$
	Using this, it is routine to show that the composition
    $$ \R \xrightarrow{\iota_\R} {G \ltimes_{\rm{lt}} \Ind_G(\R)}
        \xrightarrow{\Theta} \K(\ell\sr^2(G)) $$   
	is given by $1 \mapsto \delta_{e,e} $. It follows that $\iota_A$ maps to a full corner of 
    $G \ltimes_{\rm{lt}} \Ind_G(A)$ and thus is an isomorphism on $\Kk$-theory.
\end{proof}

\begin{proposition}\label{prop:derived-is-homology}
Let $G$ be a countable discrete group, let $A \in \KK^G$ be a separable real $G$-$\cs$-algebra and consider the homological ideal $$\mathfrak I^G = \ker (\Res_G \colon \KK^G \to \KK),$$ and the  $\mathfrak I^G$-projective resolution $P_\bullet \to A$ with $P_n = (\Ind_G \Res_G)^{n+1} A$. Then the homology of the chain complex
\begin{equation}\label{homology chain complex}
\cdots \to \KO_i(P_n)_G \to \cdots \to \KO_i(P_0)_G \to 0 
\end{equation}
computes $\Hh_n(G;\KO_i(A))$. Moreover, the chain map 
\[ \nu_{P_\bullet} \colon \KO_i(P_\bullet)_G \to \KO_i(G \ltimes P_\bullet) \]
is an isomorphism, inducing an isomorphism 
\begin{equation}\label{eq:isomorphism in homology}
\Hh_n(G;\KO_i(A)) \to \mathbb L_n^{{\mathfrak I}^G} \KO_i(G \ltimes A) 
\end{equation}
in homology. This isomorphism is natural with respect to morphisms in $\KK^G$.
\end{proposition}

\begin{proof}
For the first claim, we note that $\KO_* \colon \KK^G \to \gMod{}_*$ is a stable homological functor which vanishes on $\mathfrak I^G$. It follows from \cite[Lemma 3.9]{meyer-nest-2010} that since 
$P_\bullet \to A$ is $\mathfrak I^G$-exact,  $\KO_i(P_\bullet) \to \KO_i(A)$ is a resolution of $G$-modules. As a $G$-module, $\KO_i(P_n)$ is isomorphic to $\bigoplus_G \KO_i(R_n)$, where $R_n = \Res_G (\Ind_G \Res_G)^n A {= \Res_G P_{n-1}}$, whose $G$-action permutes the coordinates. That is, $\KO_i(P_n)$ agrees with the  induced module $\Ind_G \KO_i(R_n)$, which is $G$-acyclic by Shapiro's Lemma (cf.~\cite[Proposition III.6.2]{brownbook}). Therefore, by Proposition~\ref{prop-acyclic-res}, the chain complex \eqref{homology chain complex} computes $\Hh_n(G;\KO_i(A))$.

The inclusion $\iota_{R_n} \colon R_n \to G \ltimes P_n$ is the composition of the inclusion of $R_n$ into $P_n { = \Ind_G R_n}$ and the inclusion of $P_n$ into $G \ltimes P_n$. Consider the following commutative diagram:
\[\begin{tikzcd}
	{\KO_i(R_n)} & {\KO_i(P_n)} & {\KO_i(P_n)_G} \\
	& {\KO_i(G \ltimes P_n) }
	\arrow[bend left, "\cong"', from=1-1, to=1-3]
	\arrow[from=1-1, to=1-2]
	\arrow[from=1-2, to=1-3]
	\arrow["\KO_i(\iota_{R_n})"', "\cong",  from=1-1, to=2-2]
	\arrow[two heads, from=1-2, to=1-3]
	\arrow[from=1-2, to=2-2]
	\arrow["\nu_{P_n}", from=1-3, to=2-2]
\end{tikzcd}\]
The map $\iota_{R_n}$ induces an isomorphism in K-theory by Lemma \ref{lem:takai}. It is straightforward to show that the composed top arrow is an isomorphism under the identification $\KO_i(P_n) \cong \bigoplus_G \KO_i(R_n)$.
It follows that $\nu_{P_n}$
is an isomorphism.

For naturality, let $\alpha \colon A \to B$ be a morphism in $\KK^G$, and let $Q_\bullet \to B$ be the analogous 
$\mathfrak I^G$-projective resolution of $B$ with $Q_n = (\Ind_G \Res_G)^{n+1} B$ for $n \geq 0$. By functoriality of $\Ind_G \Res_G$, we obtain a map $f_\bullet \colon P_\bullet \to Q_\bullet$, which is a chain map over $\alpha$ by naturality of $\delta_n \colon P_n \to P_{n-1}$ for $n \geq 0$. By definition, $\mathbb L_n^{{\mathfrak I}^G} \KO_i(G \ltimes \alpha)$ is the map induced by $\KO_i(G \ltimes f_\bullet)$ in homology. Moreover, as $\KO_i(f_\bullet) \colon \KO_i(P_\bullet) \to \KO_i(Q_\bullet)$ is a chain map of $G$-acyclic resolutions over $\KO_i(\alpha)$, $\KO_i(f_\bullet)_G$ induces $\Hh_n(G;\KO_i(\alpha))$ in homology. Naturality of \eqref{eq:isomorphism in homology} now follows from naturality of $\nu$ applied to $f_\bullet$.
\end{proof}

Now, let $\Omega \colon G \to H$ be a homomorphism of countable discrete groups, and $A$ be a separable real $H$-$\cs$-algebra.
We obtain an ABC morphism \cite[Definition 5.7]{miller-hlogy}
$$(\KK^G,\mathfrak I^G, \KO_*(G \ltimes \Input),\Res^G_H A) \to (\KK^H,\mathfrak I^H, \KO_*(H \ltimes \Input),A)$$ 
as follows. The homomorphism $\Omega \colon G \to H$ induces a $*$-homomorphism 
$$\omega\colon G \ltimes \Res^G_H B \to H \ltimes  B$$ 
which is natural in $B \in \KK^H$. Let $\omega_* \colon \KO_i(G \ltimes \Res^G_H \Input) \Rightarrow \KO_i(H \ltimes \Input)$ be the induced  natural transformation. {Note that $\Res^G_H(\mathfrak I^H) \subseteq \mathfrak I^G$ because $\Res_G \Res^G_H = \Res_H \colon \KK^H \to \KK$.} Then 
\begin{equation}\label{ABC morphism}
\ABC(\Omega) := (\Res^G_H,\omega_*,\id_{\Res^G_H A}) 
\end{equation}
is the desired ABC morphism. 

Consequently, \cite[Theorem 6.1]{miller-hlogy} tells us that $\ABC(\Omega)$ induces maps 
\[ \mathbb L_n(\Omega) \colon \mathbb L_n^{{\mathfrak I}^G}(\KO_*(G \ltimes A)) \to \mathbb L_n^{{\mathfrak I}^H}(\KO_*(H \ltimes A))\]
on the $E^2$ page of the associated spectral sequences, which converge to the map 
$\mathbb L(\Omega)\colon \mathbb L^{\mathfrak I^G} \KO_*(G \ltimes A) \to \mathbb L^{\mathfrak I^H} \KO_*(H \ltimes A)$.
The next result (cf.~\cite[Theorem 5.14]{miller-hlogy} for the complex case) tells us that these maps on the $E^2$ page  are in fact the same as the maps  $\Hh_n(G; \KO_*(A)) \to \Hh_n(H; \KO_*(A)) $  induced by $\Omega\colon G \to H$. 

\begin{thm}\label{group homology theorem}
	Let $G, H$ be countable discrete groups 
	and let $\Omega\colon G \to H$ be a group homomorphism. 
	Then the isomorphism of Proposition \ref{prop:derived-is-homology} is natural with respect to $\Omega$. 
	That is, for any separable real $H$-$\cs$-algebra $A$, the diagram below commutes.
	\label{thm:group-naturality}
\begin{equation} \label{eq:diagram}
\begin{tikzcd}[column sep = 2cm]
	\Hh_n(G;\KO_i(A))  &
	 \Hh_n(H;\KO_i(A)) \\
	\mathbb L_n^{{\mathfrak I}^G} \KO_i(G \ltimes A) &
	\mathbb L_n^{{\mathfrak I}^H} \KO_i(H \ltimes A) 
	\arrow["{\cong}", from=1-1, to=2-1]
	\arrow["{\cong}", from=1-2, to=2-2]
	\arrow["{\Hh_n(\Omega;\KO_i(A))}", from=1-1, to=1-2]
	\arrow["{\mathbb L_n(\Omega)}", from=2-1, to=2-2]
\end{tikzcd}
\end{equation}
\end{thm}
\begin{proof}
Consider the ${\mathfrak I}^G$-projective resolution $P_\bullet \to \Res_H^G A$ in $\KK^G$ as in Proposition \ref{prop:derived-is-homology}, but also consider the analogous $Q_\bullet \to A$ in $\KK^H$. There exists a chain map $f \colon P_\bullet \to \Res^G_H Q_\bullet$ over $\Res^G_H A$ by \cite[Lemma 2.17]{miller-hlogy} (see also \cite[Proposition 3.26]{meyer-nest-2010}). 
Then consider the following diagram \eqref{eq:KO-diagram} of chain complexes and chain maps. For an $H$-module $N$ (in this case $N = \KO_i(Q_\bullet)$), recall the homomorphism $\mu_{N}\colon (\Res^G_H N)_G \to N_H$ introduced in the construction of $\Hh_*(\Omega; M)$ from  Proposition~\ref{prop-grouphomology-natural}.
{Also note that there is an obvious identification $\KO_i(\Res^G_H Q_\bullet)_G = (\Res^G_H \KO_i(Q_\bullet))_G$.}

\begin{equation}
\label{eq:KO-diagram}
\begin{tikzcd}[column sep = 1.5cm]
	{\KO_i(P_\bullet)_G} & 
    {\KO_i(\Res^G_H Q_\bullet)_G} & 
    {\KO_i(Q_\bullet)_H} \\
	{\KO_i(G \ltimes P_\bullet)} & 
    {\KO_i(G \ltimes \Res^G_H Q_\bullet)} & 
    {\KO_i(H \ltimes Q_\bullet)}
	\arrow["{\KO_i(f)_G}", from=1-1, to=1-2]
	\arrow["{\nu_{P_\bullet}}"', from=1-1, to=2-1]
	\arrow["{\mu_{\KO_i(Q_\bullet)}}", from=1-2, to=1-3]
	\arrow["{\nu_{Q_\bullet}}", from=1-3, to=2-3]
        \arrow["{\nu_{\Res^G_H Q_\bullet}}", from=1-2, to=2-2]
	\arrow["{\KO_i(\id_G \ltimes f)}", from=2-1, to=2-2]
	\arrow["{\omega_*}", from=2-2, to=2-3]
\end{tikzcd}\end{equation}

We observed in the proof of Proposition~\ref{prop:derived-is-homology} that  $\KO_i(P_\bullet)$ and $\KO_i(Q_\bullet)$ are acyclic resolutions of $\KO_i(A)$.  It now follows from {our discussion related to
Proposition~\ref{prop-grouphomology-natural}} that $\Hh_n(\Omega; \KO_i(A))\colon \Hh_n(G; \KO_i(A)) \to \Hh_n(H; \KO_i(A))$ is induced on homology by the top line of the diagram. As the vertical maps $\nu_*$ induce the isomorphisms $\Hh_n(G; \KO_i( - )) \to \mathbb L^{{\mathfrak I}^G}_n \KO_i(G \ltimes  - )$ by Proposition \ref{prop:derived-is-homology}, the outer diagram of \eqref{eq:KO-diagram} induces \eqref{eq:diagram} on the level of homology.  Thus, it suffices to show that the diagram \eqref{eq:KO-diagram} commutes. 
The left square commutes 
by naturality of $\nu$ with respect to morphisms in $\KK^G$. The right square also commutes 
as both compositions are ultimately induced by the inclusions $\Res^G_H Q_n \to H \ltimes Q_n$.
\end{proof}

\begin{rmk}
	In the case when $G \leq H$ and $\Omega \colon G \to H$ is the inclusion, an explicit formula for the chain map $f\colon P_\bullet \to \Res^G_H Q_\bullet$ in $\KK^G$ is given as follows. We have $P_n = (\Ind_G \Res_G)^{n+1}(\Res^G_H A) \cong C_0(G^{n+1},A)$ and $Q_n = \Res^G_H (\Ind_H \Res_H)^{n+1}(A) \cong C_0(H^{n+1},A)$; 
        \[ C_0(G^{n+1},A) \hookrightarrow C_0(H^{n+1},A) \]
    is induced by the inclusion $G^{n+1} \hookrightarrow H^{n+1}$.
        When $\Omega \colon G \to H$ is a non-injective homomorphism, one may still construct a morphism $f  \colon P_n  \to \Res^G_H Q_n$ in $\KK^G$ using the map of discrete spaces $G^{n+1} \to H^{n+1}$. Instead of an inclusion this induces a proper $\cs$-correspondence from $C_0(G^{n+1},A)$ to $C_0(H^{n+1},A)$ via a process referred to in \cite[Section 3]{PY25} as Atiyah transfer. Explicitly, $C_c(G^{n+1},A)$ is given the structure of a pre-Hilbert $C_0(H^{n+1},A)$-module with inner product
	\[ \langle \xi, \zeta \rangle (h_0, \ldots, h_n) = \sum_{(g_0, \ldots, g_n)\colon \Omega(g_i) = h_i} \xi(g_0, \ldots, g_n)^* \zeta(g_0, \ldots, g_n)\]
for $\xi, \zeta \in C_c(G^{n+1}, A)$, and the right module structure $\zeta \cdot \eta(g_0, \ldots, g_n) = \zeta(g_0, \ldots, g_n) \eta(\Omega(g_0), \ldots, \Omega(g_n))$ for $\zeta \in C_c(G^{n+1}, A)$ and $\eta \in C_0(H^{n+1}, A)$.
\end{rmk}

\section{Main Results}

\subsection{Real K-theory}

Combining the results of the previous sections, we summarize the main result for real $\Kk$-theory,  as follows. The analogous statement for complex $\cs$-algebras was proved in \cite{miller-hlogy}; see Remark \ref{rmk:main-for-complex} below.

 \begin{thm} ~ \label{KO-sp.seq.summary}
 \begin{enumerate}
  \item[\rm{(1)}]  
Let $G$ be a torsion-free countable discrete group which satisfies the Baum--Connes conjecture with coefficients in a separable real $G$-$\cs$-algebra $A$.  Then there exists a spectral sequence $E^r_{p,q} = E^r_{p,q}(G, A)$ satisfying
 \begin{itemize}
 \item $E^2_{p,q} \cong \Hh_p(G; \KO_q(A))$ and 
 \item $E^r_{p,q} \Longrightarrow \KO_{p+q}(G \ltimes_r A)$.
 \end{itemize}
 \item[\rm{(2)}]  
 Let $G$ be a torsion-free countable discrete group which satisfies the Baum--Connes conjecture with coefficients in separable real $G$-$\cs$-algebras $A$ and $B$. Let 
 $\alpha \in \KK^G(A, B)$. Then there is an induced map $\alpha_*$ on spectral sequences such that the diagram  
 \[\begin{tikzcd}
 E^2_{p,q}(G,A) \arrow[r, "\cong"] \arrow[d, "\alpha_*"] 
 	& \Hh_p(G; \KO_q(A)) \arrow[d, "{\Hh_p(G;\KO_q(\alpha))}"]  \\
 E^2_{p,q}(G,B) \arrow[r, "\cong"]
 	& \Hh_p(G; \KO_q(B)) 
 \end{tikzcd} \]
 commutes and $\alpha_*$ converges to $\KO_{p+q}(G\ltimes_r \alpha)$:
 \[ \begin{tikzcd}
 {E^r_{p,q}(G, A) } \arrow[r, Rightarrow] \arrow[d, "\alpha_* "] 
 	& {\KO_{p+q}(G \ltimes_r A)} \arrow[d, "\KO_{p+q}(G \ltimes_r \alpha) "]  \\
{E^r_{p,q}(G, B)}  \arrow[r, Rightarrow] 
 	& {\KO_{p+q}(G \ltimes_r B) }
\end{tikzcd} \]
Moreover, the assignment $\alpha \mapsto \alpha_*$ is functorial. 
   \item[\rm{(3)}]  
 Let $\Omega \colon G \rightarrow H$ be a homomorphism of torsion-free countable discrete groups which satisfy the Baum--Connes conjecture with coefficients in a separable real $H$-$\cs$-algebra $A$ (viewing $A$ as a $G$-$\cs$-algebra via $g \cdot a := \Omega(g) \cdot a$). Then there is an induced morphism  $\ABC(\Omega)_*$  of spectral sequences such that the diagram  
\[\begin{tikzcd}
 E^2_{p,q}(G,A) \arrow[r, "\cong"] \arrow[d, "\ABC(\Omega)_*"] 
 	& \Hh_p(G; \KO_q(A)) \arrow[d, "{\Hh_p(\Omega; \KO_q(A))}"]  \\
 E^2_{p,q}(H,A) \arrow[r, "\cong"]
 	& \Hh_p(H; \KO_q(A)) 
 \end{tikzcd} \]
 commutes and $\ABC(\Omega)_*$ converges to $\KO^{\mathrm{top}}_{p+q}(\Omega;A)$:
  \[\begin{tikzcd}
 {E^r_{p,q}(G, A) } \arrow[r, Rightarrow] \arrow[d, "\ABC(\Omega)_*"] 
 	& {\KO_{p+q}(G \ltimes_r A)} \arrow[d, "{\KO^{\mathrm{top}}_{p+q}(\Omega;A)}"]  \\
{E^r_{p,q}(H, A)}  \arrow[r, Rightarrow] 
 	& {\KO_{p+q}(H \ltimes_r A) }
 \end{tikzcd} \]
 The map $\KO^{\mathrm{top}}_{p+q}(\Omega;A)$ is induced by $\Omega$ on the level of topological K-theory, which is identified with the K-theory of the reduced crossed product via the Baum--Connes assembly map.
 In particular, if $\ker \Omega$ is amenable, $\KO^{\mathrm{top}}_{p+q}(\Omega;A)$ is induced by the  $*$-homomorphism $\Omega \ltimes_r \id_A \colon G \ltimes_r A \to H \ltimes_r A$.
 
  \item[\rm{(4)}]  Let $\Omega \colon G \rightarrow H$ be a homomorphism of torsion-free countable discrete groups which satisfy the Baum--Connes conjecture with coefficients in separable real $H$-$\cs$-algebras $A$ and $B$, and let $\alpha \in \KK^H(A, B)$ where $A,B$ are separable real $H$-$\cs$-algebras. Then the following diagram of spectral sequences commutes.
\[ \begin{tikzcd}
 {E^r_{p,q}(G, A) } \arrow[r, "\alpha_* "] \arrow[d, "\ABC(\Omega)_*"] 
 	& {E^r_{p,q}(G, B) } \arrow[d, "\ABC(\Omega)_*"]  \\
{E^r_{p,q}(H, A)}  \arrow[r, "\alpha_* "] 
 	& {E^r_{p,q}(H, B)  }
 \end{tikzcd} \]
 \end{enumerate}
 \end{thm}

 \begin{proof} 
 Statement (1) follows directly from Propositions \ref{prop:ABC-spectral} and \ref{prop:derived-is-homology}. 
Statement (2) describes the naturality of (1) and follows directly from the naturality within Propositions \ref{prop:ABC-spectral} and \ref{prop:derived-is-homology}.

For (3), we apply Theorem \ref{thm:group-naturality} to obtain the commutative diagram, also using that $E^2_{p,q} = \mathbb L_p^{\mathfrak I^G}(\KO_q(G \ltimes A))$ from Proposition \ref{prop:ABC-spectral}. We define $\KO^{\mathrm{top}}_{p+q}(\Omega;A)$ to be the map $\mathbb L(\Omega) \colon \mathbb L^{\mathfrak I^G} \KO_*(G \ltimes A) \to \mathbb L^{\mathfrak I^H} \KO_*(H \ltimes A)$ induced by $\ABC(\Omega)$, identifying these localizations with topological K-theory as in the proof of Proposition \ref{prop:ABC-spectral}. That $\ABC(\Omega)_*$ converges to $\KO^{\mathrm{top}}_{p+q}(\Omega;A)$ now follows from the functoriality of the ABC spectral sequence \cite[Theorem 6.1]{miller-hlogy}.
 For the `in particular' statement, note the construction of $\Omega \ltimes_r \id_A$ from Proposition \ref{prop:amenable-kernel}. Consider the following diagram. \[\begin{tikzcd}
 \KO^{\mathrm{top}}_{p+q}(G;A) \arrow[r] \arrow[d, "{\KO^{\mathrm{top}}_{p+q}(\Omega;A)}"] & \KO_{p+q}(G \ltimes A) \arrow[r] \arrow[d, "{\KO_{p+q}(\Omega \ltimes \id_A)}"] 
 	& \KO_{p+q}(G \ltimes_r A) \arrow[d, "{\KO_{p+q}(\Omega \ltimes_r \id_A)}"]  \\
 \KO^{\mathrm{top}}_{p+q}(H;A) \arrow[r] & \KO_{p+q}(H \ltimes A) \arrow[r]
 	& \KO_{p+q}(H \ltimes_r A) 
 \end{tikzcd} \]
 The first square commutes by naturality of the localization map with respect to $\ABC(\Omega)$ (see \cite[Definition 5.10]{miller-hlogy}), and the second square commutes by construction. The compositions of the horizontal maps are the reduced assembly maps, which thereby identify $\KO^{\mathrm{top}}_{p+q}(\Omega;A)$ with $\KO_{p+q}(\Omega \ltimes_r \id_A)$.
  Finally, Statement (4) is a consequence of the functoriality of the ABC spectral sequence, that an ABC morphism functorially induces a morphism of ABC spectral sequences \cite[Theorem 6.1]{miller-hlogy}.
 \end{proof}

 	\begin{rmk}
 		\label{rmk:main-for-complex}
 	{The analogue of this result for complex $\cs$-algebras also holds. 
    This can be obtained directly from Theorem~\ref{KO-sp.seq.summary} by noting that when $A$ is a complex $\cs$-algebra, we have
    $\KO_*(A) = \Kk_*(A),$
where the left side involves forgetting the complex structure of $A$.
}

        We can also extract the result for complex $\cs$-algebras from \cite{miller-hlogy} as follows.
 	Consider a homomorphism $\Omega \colon G \to H$ of countable torsion-free discrete groups satisfying the Baum--Connes conjecture with coefficients, and a complex $\cs$-algebra $A$ with an action of $H$. By combining \cite[Example 2.12]{miller-hlogy},
 	\cite[Example 5.8]{miller-hlogy}  and \cite[Example 3.2]{miller-categories}, we obtain the ABC morphism 
    \[ \ABC_{\mathbb C}(\Omega) \colon (\KK^G_{\mathbb C}, \mathfrak I^G_{\mathbb C}, \Kk_*(G \ltimes -), A) \to (\KK^H_{\mathbb C}, \mathfrak I^H_{\mathbb C}, \Kk_*(H \ltimes -), A)\] 
    analogous to \eqref{ABC morphism}. By \cite[Theorem 6.1]{miller-hlogy} this induces a morphism of the associated ABC spectral sequences, given on the $E^2$ page by the map of derived functors $\mathbb L_n(\Omega) = \mathbb L_n(\ABC_{\mathbb C}(\Omega))$ and converging to the map of target groups given by the localization map $\mathbb L(\Omega) = \mathbb L(\ABC_{\mathbb C}(\Omega))$. By \cite[Example 2.20 and Example 5.11]{miller-hlogy}, $\mathbb L(\Omega)$ identifies with a map $\Kk_*(G \ltimes_r A) \to \Kk_*(H \ltimes_r A)$, and \cite[Theorem 5.14]{miller-hlogy} identifies $\mathbb L_n(\Omega)$ with the map $$\Hh_n(\Omega;\Kk_*(A)) \colon \Hh_n(G; \Kk_*(A)) \to \Hh_n(H; \Kk_*(A))$$ in group homology induced by $\Omega$.
 	\end{rmk}

\subsection{United K-theory}\label{subsection-unitedK}

United $\Kk$-theory is a functor on $\KK$ that includes real $\Kk$-theory $\KO_*(A)$, together with the so-called complex and self-conjugate $\Kk$-theory, as well as natural transformations between them. It was introduced in \cite{bousfield90} in the context of topological $\Kk$-theory and in \cite{boersema2002} for real $\cs$-algebras. 
One advantage of united $\Kk$-theory is that it is a more complete invariant. For example, $\KO_*(A)$ by itself does not classify real $\cs$-algebras in the real bootstrap category up to $\KK$-equivalence, but united $\Kk$-theory does so by \cite{boersema2004}. 
United $\Kk$-theory also completely classifies up to isomorphism real Kirchberg algebras in the real bootstrap category (either unital or stable) by \cite{brs}.
Furthermore, it is generally easier to consider real and complex $\Kk$-theory together as part of united $\Kk$-theory because the additional structure serves to facilitate computation (cf.~\cite{boersema-vdovina} and \cite[Section 4]{boersema-gillaspy}). For these reasons, we will articulate our main theorem in terms of united $\Kk$-theory, below.

First we give a brief summary of united $\Kk$-theory in the context of real $\cs$-algebras, but see the references above for a full development.
For a real $\cs$-algebra $A$, ``United $\Kk$-theory" refers to one of the following functors: \begin{align*}
	\Kk\crr(A) &= \{ \KO_*(A) , \KU_*(A), \Theta\crr \}  \\
	\text{or} \quad \Kk\crt(A) &= \{ \KO_*(A) , \KU_*(A), \KT_*(A), \Theta\crt \} 
\end{align*}
where
\begin{align*}
	\KO_*(A) &= \text{the standard operator algebra period-8 $\Kk$-theory of $A$,} \\
	\KU_*(A) &= \Kk_*(A\sc) = \text{the standard complex operator algebra period-2 $\Kk$-theory of $A\sc$}, \\
	\KT_*(A) &= \Kk_*(A \otimes T) \text{~where $T = \{f \colon [0,1] \rightarrow \C \text{ continuous } \mid f(0) = \overline{f(1)} \} $ }, \\
	\Theta\crr &= \text{the natural transformations among $\KO_*$ and $\KU_*$ 
    }, \\	
	\Theta\crt &=  \text{the natural transformations among $\KO_*$, $\KU_*$, and $\KT_*$} . 
\end{align*}

The third group $\KT_*(A)$, known as `self-conjugate $\Kk$-theory', was introduced in \cite{atiyah} in the topological setting and defined as above in \cite{boersema2002} for real $\cs$-algebras.

The full version $\Kk\crt(A)$ of united $\Kk$-theory is necessary for the most general versions of the K\"unneth Formula and the Universal Coefficient Theorem in \cite{boersema2002} and \cite{boersema2004}.  However, the smaller version $\Kk\crr(A)$ of united $\Kk$-theory  is sufficient for many purposes since the forgetful functor $\Kk\crt(A) \mapsto \Kk\crr(A)$ is injective on isomorphism classes by the main result of \cite{Hewitt}.

The natural transformations $\Theta\crt$ arise from taking the Kasparov product with the elements of $\KK_i(X,Y)$ for $X,Y \in \{\R,\C,T\}$ and $i\in \mathbb Z$. %
They satisfy certain relations $\mathrm{Rel}\crt$ due to bilinearity and associativity of the Kasparov product, which say that they are compatible with addition in $\KK_i(X,Y)$ and the Kasparov product $\otimes \colon \KK_i(X,Y) \times \KK_j(Y,Z) \to \KK_{i+j}(X,Z)$. Abstractly, the united K-theory $\Kk\crt(A)$ of a real $\cs$-algebra $A$ has the structure of a CRT-module:

\begin{defn}
A \emph{$\CRT$-module} $M = (M\po, M\pu, M\pt, \Theta\crt_M)$ consists of 
\begin{itemize}
\item a triple $(M\po, M\pu, M\pt)$ of $\mathbb Z$-graded abelian groups, together with 
\item a collection 
\[\Theta\crt_M = \{ M_\theta \mid i \in \mathbb Z, \, X,Y \in \{\R,\C,T\}, \, \theta \in \KK_i(X,Y) \}\] of homomorphisms $M_{\theta} \colon M\px_* \to M\py_{*+i}$ of degree $i$, compatible with addition in the Kasparov groups and with the Kasparov product. Explicitly, we require that $M_{\theta_1 + \theta_2} = M_{\theta_1} + M_{\theta_2}$ for $\theta_1,\theta_2 \in \KK_i(X,Y)$ and $M_{\theta_3 \otimes \theta_4} =  M_{\theta_4} \circ M_{\theta_3}$ for $\theta_3 \in \KK_i(X,Y)$ and $\theta_4 \in \KK_j({Y},Z)$.
\end{itemize}
A {\em morphism of $\CRT$-modules} is a triple of graded homomorphisms compatible with the natural transformations $M_\theta$. 
\end{defn}

The natural transformations $\theta \in \KK_i(X,Y)$ in $\Theta\crt$ and 
their relations $\mathrm{Rel}\crt$ are described in detail in \cite{boersema2002} and \cite{bousfield90}.

Our spectral sequence(s) which converge to the CRT-module $\Kk\crt(G \ltimes_r A)$ also have extra CRT-module-like structure, which we encode in the following definition.

\begin{defn}
A \emph{$\CRT$-spectral sequence} $$(E^r_{p,q})\crt = ((E^r_{p,q})\po, (E^r_{p,q})\pu, (E^r_{p,q})\pt, \Theta\crt_E)$$ consists of 
\begin{itemize}
\item a triple $((E^r_{p,q})\po, (E^r_{p,q})\pu, (E^r_{p,q})\pt)$ of spectral sequences, together with
\item a collection \[\Theta\crt_E = \{ f_\theta \mid i \in \mathbb Z, \, X,Y \in \{\R,\C,T\}, \, \theta \in \KK_i(X,Y) \}\] of morphisms $f_\theta \colon (E^r_{*,*})\px \to (E^r_{*,*+i})\py$ of spectral sequences of degree $(0,i)$, compatible with addition in the Kasparov groups and with the Kasparov product. 
\end{itemize}
A {\em morphism of $\CRT$-spectral sequences} is a triple of morphisms of spectral sequences compatible with the natural transformations $f_\theta$.
\end{defn}
In particular, each column $((E^r_{p,*})\po, (E^r_{p,*})\pu, (E^r_{p,*})\pt, \Theta\crt_E)$ of a CRT-spectral sequence is a CRT-module,
and each differential 
$$d^r \colon (E^r_{p,*})\px \rightarrow (E^r_{p-r,*+r-1})\px$$
is a morphism of CRT-modules.

We are now prepared to state and prove our main result.
  
  \begin{thm} ~ \label{CRT-sp.seq.summary}
 \begin{enumerate}
 \item[\rm{(1)}]  
Let $G$ be a torsion-free countable discrete group which satisfies the Baum--Connes conjecture with coefficients in a separable real $G$-$\cs$-algebra $A$.  Then there exists a~$\CRT$-spectral sequence $(E^r_{p,q})\crt = (E^r_{p,q})\crt(G, A)$ satisfying
 \begin{itemize}
 \item $(E^2_{p,q})\crt \cong \Hh_p(G; \Kk\crt_q(A))$ and
 \item $(E^r_{p,q})\crt \Longrightarrow \Kk\crt_{p+q}(G \ltimes_r A)$.
 \end{itemize}
 \item[\rm{(2)}]
 Let $G$ be a torsion-free countable discrete group which satisfies the Baum--Connes conjecture with coefficients in separable real $G$-$\cs$-algebras $A$ and $B$. Let $\alpha \in \KK^G(A, B)$. Then there is an induced morphism $\alpha_*$ of $\CRT$-spectral sequences
 such that the diagram
 \[ \begin{tikzcd}
 {(E^r_{p,q})\crt(G, A) } \arrow[r, "\cong"] \arrow[d, "\alpha_*"] 
 	& \Hh_p(G; \Kk\crt_q(A)) \arrow[d, "{\Hh_p(G;\Kk\crt_q(\alpha))}"]  \\
 {(E^r_{p,q})\crt(G, B) } \arrow[r, "\cong"]
 	& \Hh_p(G; \Kk\crt_q(B)) \\
 \end{tikzcd} \]
commutes and $\alpha_*$ converges to $\Kk\crt_{p+q}(G \ltimes_r \alpha) $:
\[ \begin{tikzcd}
 {(E^r_{p,q})\crt(G, A) } \arrow[r, Rightarrow] \arrow[d, "\alpha_* "] 
 	& {\Kk\crt_{p+q}(G \ltimes_r A)} \arrow[d, "\Kk\crt_{p+q}(G \ltimes_r \alpha) "]  \\
{(E^r_{p,q})\crt(G, B)}  \arrow[r, Rightarrow] 
 	& {\Kk\crt_{p+q}(G \ltimes_r B) }
\end{tikzcd} \]
  \item[\rm{(3)}]
 Let $\Omega \colon G \rightarrow H$ be a homomorphism of torsion-free countable discrete groups which satisfy the Baum--Connes conjecture with coefficients in a separable real $H$-$\cs$-algebra $A$ (viewing $A$ as a $G$-$\cs$-algebra via $g \cdot a := \Omega(g) \cdot a$). Then there is an induced morphism $\mathrm{ABC}(\Omega)_*$ of $\CRT$-spectral sequences such that the diagram
 \[ \begin{tikzcd}
 (E^2_{p,q})\crt(G,A) \arrow[r, "\cong"] \arrow[d, "\mathrm{ABC}(\Omega)_*"] 
 	& \Hh_p(G; \Kk\crt_q(A)) \arrow[d, "{\Hh_p(\Omega; \Kk\crt_q(A))}"]  \\
 (E^2_{p,q})\crt(H,A) \arrow[r, "\cong"]
 	& \Hh_p(H; \Kk\crt_q(A)) \\
 \end{tikzcd} \]
 commutes and $\ABC(\Omega)_*$ converges to $\Kk\crt_{p+q}(\Omega; A)$: 
   \[ \begin{tikzcd}
 {(E^r_{p,q})\crt(G, A) } \arrow[r, Rightarrow] \arrow[d, "\mathrm{ABC}(\Omega)_*"] 
 	& {\Kk\crt_{p+q}(G \ltimes_r A)} \arrow[d, "{\Kk^{\scriptscriptstyle{\mathrm{top}, \CRT}}_{p+q}(\Omega;A)}"]  \\
{(E^r_{p,q})\crt(H, A)}  \arrow[r, Rightarrow] 
 	& {\Kk\crt_{p+q}(H \ltimes_r A) }
 \end{tikzcd} \]
 The map $\Kk^{\scriptscriptstyle{\mathrm{top}, \CRT}}_{p+q}(\Omega;A)$ is induced by $\Omega$ on the level of topological K-theory, which is identified with the K-theory of the reduced crossed product via the Baum--Connes assembly map.
 In particular, if $\ker \Omega$ is amenable,
 $\Kk^{\scriptscriptstyle{\mathrm{top}, \CRT}}_{p+q}(\Omega;A)$ is induced by the resulting $*$-homomorphism $\Omega \ltimes_r \id_A \colon G \ltimes_r A \to H \ltimes_r A$.    \item[\rm{(4)}]  Let $\Omega \colon G \rightarrow H$ be a homomorphism of torsion-free countable discrete groups which satisfy the Baum--Connes conjecture with coefficients in separable real $H$-$\cs$-algebras $A$ and $B$, and let $\alpha \in \KK^H(A, B)$. 
 Then the following diagram of $\CRT$-spectral sequences commutes.
\[ \begin{tikzcd}
 {(E^r_{p,q})\crt(G, A) } \arrow[r, "\alpha_* "] \arrow[d, "\ABC(\Omega)_*"] 
 	& {(E^r_{p,q})\crt(G, B) } \arrow[d, "\ABC(\Omega)_*"]  \\
{(E^r_{p,q})\crt(H, A)}  \arrow[r, "\alpha_* "] 
 	& {(E^r_{p,q})\crt(H, B)  }
 \end{tikzcd} \]
 \end{enumerate}
 \end{thm}
 
 \noindent
  
 \begin{proof}
 Let $A$ be a separable real $\cs$-algebra with a $G$-action. Then $A\sc = A \otimes \C$ and $A \otimes T$ also have an inherited $G$-action.  Furthermore, there are isomorphisms $G \ltimes_r (A \otimes \C) \cong (G \ltimes_r A) \otimes \C$ and
  $G \ltimes_r (A \otimes T) \cong (G \ltimes_r A) \otimes T$.
  Therefore, applying Part (1) of Theorem~\ref{KO-sp.seq.summary} 
  to all three of these algebras, we obtain three spectral sequences:
  \begin{align*} 
  (E^r_{p,q})\po &= E^r_{p,q}(G, A)  \numberthis \label{3bears} \\
 (E^r_{p,q})\pu &= E^r_{p,q}(G, A \otimes \C) \\
 (E^r_{p,q})\pt &= E^r_{p,q}(G, A \otimes T)
  \end{align*}
  We will write these as  $(E^r_{p,q})\px = E^r_{p,q}(G, A \otimes X)$ for $X \in \{O, U, T\} = \{\R, \C, T\}$. In each case, the $E^2$ page is given by
  $$(E^2_{p,q})\px = \Hh_p(G; \KO_q(A \otimes X)) 
  	= \Hh_p(G; \KX_q(A))\; $$
and the spectral sequence converges according to the rule
$$(E^r_{p,q})\px \Longrightarrow \KO_{p+q}(G \ltimes_r (A \otimes X))
	= \KX_{p+q}(G \ltimes_r A) \; .$$

Furthermore, suppose that a natural transformation $\theta \in \Theta\crt$ is given by $\theta \in \KK_i(X, Y)$ for $X,Y \in \{\R, \C, T\}$.
Then for any separable real $G$-$\cs$-algebra $A$, $\theta$ functorially induces an element $\theta_A \in \KK_i^G(A \otimes X, A \otimes Y)$. 
By Part (2) of Theorem~\ref{KO-sp.seq.summary}, $\theta_A$ in turn functorially induces a morphism $f_{\theta_A}$ of spectral sequences of degree $(0,i)$
$$(E^r_{p,q})\px(G, A) \rightarrow (E^r_{p,q+i})\py(G,A) \; $$
converging to the map 
$$(\theta_A)_*\colon \KX_*(G \ltimes_r A) \rightarrow \KY_{*+i}(G \ltimes_r A) .$$

Therefore, the three spectral sequences $(E^r_{p,q})\px$ 
comprise a single $\CRT$-spectral sequence
$$(E^r_{p,q})\crt(G,A) = ( (E^r_{p,q})\po(G,A), (E^r_{p,q})\pu(G,A), (E^r_{p,q})\pt(G,A) , \Theta\crt({G, A})) ,$$
where $\Theta\crt({G, A})$ consists of the maps 
$f_{\theta_A}$ described above for each 
$\theta \in \Theta\crt$. 
This $\CRT$-spectral sequence satisfies
\begin{itemize}
 \item $(E^2_{p,q})\crt(G,A) \cong \Hh_p(G; \Kk\crt_q(A))$ and 
 \item $(E^r_{p,q})\crt(G,A)  \Longrightarrow \Kk\crt_{p+q}(G \ltimes_r A)$
 \end{itemize}
because these statements hold for each of the component spectral sequences by Theorem \ref{KO-sp.seq.summary}(1) and the elements of $\Theta\crt(G, A)$ are all induced from elements of $\KK_i(X,Y)$.   This proves  Part (1) of Theorem~\ref{CRT-sp.seq.summary}.
 
 Let $\alpha \in \KK^G(A,B)$. By applying Part (2) of Theorem~\ref{KO-sp.seq.summary} for $X \in \{\R, \C, T\}$, we get maps $\alpha_*\px$ from each of the three spectral sequences $(E^r_{p,q}(G,A))\px$ to the corresponding spectral sequence $(E^r_{p,q}(G,B))\px$.
Then the triple
$$\alpha_* = \{ \alpha_*\po, \alpha_*\pu, \alpha_*\pt\} \colon (E^r_{p,q})\crt(G,A) \rightarrow (E^r_{p,q})\crt(G,B)$$
is a morphism of CRT-spectral sequences.
That the constituent elements of $\alpha_*$ commute appropriately with all of the natural transformations $\theta \in \Theta\crt$ can be seen from the fact that the diagram
\[ \begin{tikzcd}
 {\KO_*(A \otimes X) } \ar[d] \arrow[r] & {\KO_*(B \otimes X) }  \ar[d] \\
 {\KO_*(A \otimes Y) } \arrow[r] & {\KO_*(B \otimes Y) } 
 \end{tikzcd} \]
commutes, in which the horizontal maps are induced by $\alpha \in \KK^G(A,B)$ and the vertical maps are induced by $\theta \in \KK(X,Y)$.

 Theorem \ref{KO-sp.seq.summary}(2) then implies that the diagrams in Part (2) of Theorem \ref{CRT-sp.seq.summary} commute, as desired.

For Part (3), let $\Omega \colon G \rightarrow H$ be a homomorphism of torsion-free countable discrete groups which satisfy the Baum--Connes conjecture, and let $A$ be a separable real $H$-$\cs$-algebra. Then Part (3) of Theorem \ref{KO-sp.seq.summary} guarantees morphisms of  spectral sequences 
$$\ABC(\Omega)_*\px \colon (E^r_{p,q})\px(G,A) \to  (E^r_{p,q})\px(H,A)$$
for $X \in \{\R, \C, T\}$. Furthermore, for a natural transformation $\theta \in \Theta\crt$ given by
 $\theta \in \KK_i(X, Y)$ for $X,Y \in \{\R, \C, T\}$, we have a commutative diagram
\[ \begin{tikzcd}
 {E^r_{p,q}(G, A \otimes X) } \arrow[rr, "(\id_A \otimes \theta)_* "] \arrow[d, "\ABC(\Omega)_*\px"] 
 	&& {E^r_{p,q}(G, A \otimes Y) } \arrow[d, "\ABC(\Omega)_*\py "]  \\
{E^r_{p,q}(H, A \otimes X)}  \arrow[rr, "(\id_A \otimes \theta)_* "] 
 	&& {E^r_{p,q}(H, A \otimes Y)  }
 \end{tikzcd} \]
by Part (4) of Theorem~\ref{KO-sp.seq.summary}. This implies that 
$$\ABC(\Omega)_* = \{ \ABC(\Omega)_*\po, \ABC(\Omega)_*\pu, \ABC(\Omega)_*\pt \}$$ is a morphism of ${\CRT}$-spectral sequences as desired.  The commutativity of the first diagram of Part (3) follows since the three component diagrams commute by Theorem \ref{KO-sp.seq.summary}(3).

Moreover, $\ABC(\Omega)_*$  converges to
$$\Kk^{\scriptscriptstyle{\mathrm{top}, \CRT}}_{p+q}(\Omega;A)_* \colon \Kk\crt_*(G \ltimes_r A) \rightarrow \Kk\crt_*(H \ltimes_r A)$$ since each of the component spectral sequence morphisms $\ABC(\Omega)_*\px$ converges to $\KO^{\scriptscriptstyle{\mathrm{top}}}_{p+q}(\Omega;A \otimes X)$ by Theorem \ref{KO-sp.seq.summary}(3).  The description of $\ABC(\Omega)_*$ in terms of $\Kk^{\mathrm{top}, \CRT}_{p+q}(\Omega; A)$ and the `in particular' statement follow as in the proof of Theorem \ref{KO-sp.seq.summary}.

For Part (4), the existence  of  
$\alpha_*$ and $\ABC(\Omega)_*$ follows from Parts (2) and (3), and the commutativity of the diagram follows from Part (4) of Theorem~\ref{KO-sp.seq.summary}.
 \end{proof}
 
 \appendix
 \section{K-theory for $\boldsymbol{\Z^k}$ crossed products via spectral sequences: a historical overview}
 \label{sec:appendix}

 For crossed product $\cs$-algebras, two crucial tools for computing $\Kk$-theory are the well-known Pimsner--Voiculescu long exact sequence for crossed products by $\Z$ (see \cite{pims-voic, boersema-rmj}) and a spectral sequence for crossed products by $\Z^k$, which originates from \cite[Theorem 6.1]{kasparov-equivKK} and  comes in both a homological version and a cohomological version. For more on this spectral sequence and its applications, see also \cite{robertson-steger-2001, Barlak15, Barlak-Omland-Stammeier, deeley, boersema-vdovina, Lippert25}. In particular, this spectral sequence underlies the spectral sequences used in \cite{evans, boersema-gillaspy} to compute the $\Kk$-theory of higher rank graph $\cs$-algebras.
 
{In this Appendix, we discuss Kasparov's spectral sequence and some challenges it faces in concrete applications, and compare it to the ABC spectral sequence which has been the focus of this paper.}

 Briefly, for a real $\cs$-algebra $A$ with an automorphism $\alpha$, the Pimsner--Voiculescu exact sequence
 $$ \cdots \rightarrow \KO_*(A) \xrightarrow{1-\alpha_*} \KO_*(A) \xrightarrow{\iota} \KO_*(\Z \ltimes_\alpha A)
 \rightarrow \KO_*(A) \xrightarrow{1-\alpha_*} \KO_*(A) \rightarrow \cdots$$
 allows one to compute (at least up to extension) the real $\Kk$-theory $\KO_*(\Z \ltimes_\alpha A)$ of the crossed product $\Z \ltimes_\alpha A$ 
 in terms of the $\Kk$-theory of $A$ and the action of $\alpha$ thereon.
 
 For an action $\alpha$ of $\Z^k$ on $A$ (or equivalently, for $k$ commuting automorphisms $\alpha_i$ of $A$) the Kasparov spectral sequence operates similarly in the sense that the initial data is the $\Kk$-theory of $A$ and the action of $\alpha_i$ for each $i$, and the conclusion is information about the $\Kk$-theory of the crossed product 
 $\Z^k \ltimes_\alpha A$.
 More specifically, the $E^2$-page of the Kasparov spectral sequence is given by
 $$E^2_{p,q} = \Hh_{p}(\Z^k, \KO_q(A)) \; $$
 where the action of $\alpha_i$ on $\KO_*(A)$ informs the differentials of the chain complex used to compute this homology. The further differentials $d^r$ needed to determine the subsequent pages of the spectral sequence can be rather mysterious, but we know that the final page is realized as
 $E^{k+1}_{p,q} = E^\infty_{p,q}$
 (\cite[Theorem 2.1]{schochet-I}, \cite[Lemma 3.3]{evans}). 
 This spectral sequence converges to $\KO_*(\Z^k \ltimes_\alpha A)$ in the sense that there is a filtration
 $$ 0 = F_{i,0} \subseteq F_{i,1} \subseteq \cdots \subseteq F_{i,k} = \KO_i(\Z^k \ltimes_\alpha A)$$
 and there are isomorphisms $E^\infty_{p,q} \cong F_{p+q,p}/F_{p+1,p-1}$. Thus thespectral sequence allows one to compute $\KO_*(\Z^k \ltimes_\alpha A)$ up to extensions.
 In fact, this spectral sequence generalizes the Pimsner--Voiculescu  exact sequence since the  latter can be recovered from the spectral sequence in the case $k = 2$.
 
 There are two main difficulties with the spectral sequence,  for both real and complex $\cs$-algebras. First, although the $E^2$-page is readily identified, it
 does not determine the subsequent pages without knowing the action of the differentials $d^r$ and there is no known description of these in terms of $\KO_*(A)$ and $\alpha_*$. Consequently, for $r > 2$, there is no general algorithm for computing the  $E^r$ pages of the spectral sequence.  Moreover, even when the $E^\infty$ page can be found, it does not uniquely determine $\KO_*(\Z^k \ltimes_\alpha A)$ in general.
 
 Beyond the setting of $\Z^k$, for any countable discrete torsion-free group $G$ satisfying the Baum--Connes conjecture with coefficients in a separable real $\cs$-algebra $A$, 
 Kasparov's work \cite[Theorem 6.10]{kasparov-equivKK} gives a spectral sequence converging to $\KO_*(G \ltimes_{\alpha,r} A)$ that satisfies $$E^2_{p,q} = \Hh_{p}(G; \KO_q(A)). $$

 An alternative to Kasparov's approach is the ABC spectral sequence, developed in \cite{meyer-2008}. The ABC spectral sequence construction applies to a much more general setting, but for a countable discrete torsion-free group $G$ which satisfies the Baum--Connes conjecture with coefficients in $A$, it gives a spectral sequence which has the same properties as Kasparov's spectral sequence described above: it also satisfies 
 $$E^2_{p,q} = \Hh_{p}(G; \KO_q(A)) \; $$
 and it also converges to $\KO_*(G \ltimes_{\alpha,r} A)$.
 It seems likely to us that the spectral sequence arising from these two different sources are the same in the fullest sense, including that all the differentials are the same and that the associated filtrations of $\KO_*(G \ltimes_{\alpha,r} A)$ are the same, but we do not know this.
 As we have seen in earlier sections, the ABC framework has enabled us to obtain our desired naturality results about the functoriality of the spectral sequence with respect to group homomorphisms.

 {The naturality of the Kasparov spectral sequence (and also the Pimsner--Voiculescu exact sequence) with regards to morphisms of real $\cs$-algebras is well known (cf.~\cite[Remark A.2]{Barlak15}).}
  That is, if $(A, \alpha, G)$ and $(B, \beta, G)$ are $C \sp *$-dynamical systems, and $\phi \colon A \rightarrow B$ is a real $\cs$-algebra homomorphism such that $\phi \circ \alpha_g = \beta_g \circ \phi$ for all $g \in G$, then there is a homomorphism from the Kasparov spectral sequence for $(A, \alpha, G)$ to the spectral sequence for $(B, \beta, G)$ that converges to the homomorphism 
 $$\phi_* \colon \KO_*(G \ltimes_\alpha A) \rightarrow \KO_*(G \ltimes_\alpha B) .$$
 This means in particular that there is a map 
 $$\varepsilon^r_{p,q} \colon E^r_{p,q}(A) \rightarrow E^r_{p,q}(B)$$ 
 on each group of the spectral sequence that commutes with the differentials. Thus $\varepsilon^{1}$ determines $\varepsilon^{r}$ for all $r$. Furthermore, this means the map
 $\phi_*$ respects the filtrations on $\KO_*(G \ltimes_\alpha A)$ and $\KO_*(G \ltimes_\alpha B)$ and that the induced map on the subquotients of these filtrations agrees with the map $\varepsilon^\infty_{p,q}$ on $E^\infty_{p,q}$. This naturality follows fairly directly from the construction of the spectral sequence.
 
 {However, as we discussed in the introduction, naturality of the Kasparov spectral sequence with regards to group homomorphisms $\Omega\colon G\to H$ has not yet been established in the literature. 
 In contrast, the established functoriality of the ABC spectral sequences has enabled us to prove in this paper that the ABC spectral sequences are indeed natural in $G$.

\bibliography{eagbib} 
\bibliographystyle{amsalpha}

\end{document}